\documentclass[12pt]{amsart}
\usepackage{fullpage}

\usepackage{amssymb}

\renewcommand{\d}{\hat{\delta}}

\newtheorem{theorem}{Theorem}

\newtheorem{claim}{Claim}
\newtheorem{case}{Case}

\renewcommand{\S}{\mathcal{S}}

\title{Large rainbow matchings in large graphs}
\author[A.~Kostochka]{Alexandr  Kostochka}
\address[A.~Kostochka]{University of Illinois at Urbana--Champaign\\ Urbana, IL 61801\\ USA\\ and
 Sobolev Institute of Mathematics\\ Novosibirsk 630090\\ Russia} 
 \email{kostochk@math.uiuc.edu}
\thanks{The research of the first author
is supported in part by NSF grant DMS-0965587 and by
the Ministry of education and science of the Russian Federation (Contract no. 14.740.11.0868). }
\author[F.~Pfender]{Florian Pfender}
\address[F.~Pfender]{Universit\"at Rostock\\ 18055 Rostock\\ Germany\\ and University of Colorado Denver\\ CO 80202\\ USA} \email{florian.pfender@uni-rostock.de}
\author[M.~Yancey]{Matthew Yancey}
\address[M.~Yancey]{Department of Mathematics\\ University of Illinois\\ Urbana,
IL 61801\\ USA} 
\email{yancey1@illinois.edu}
\thanks{The research of the third author is partially supported by the
Arnold O. Beckman Research Award of the University of Illinois
at Urbana-Champaign.}

\begin{document}

\begin{abstract}
A \textit{rainbow subgraph} of an edge-colored graph is a subgraph whose edges have distinct colors.
The \textit{color degree} of a vertex $v$ is the number of different colors on edges incident to $v$.
We  show that if $n$ is large enough (namely, $n\geq 4.25k^2$), then each $n$-vertex graph $G$
with minimum color degree at least $k$ contains a rainbow
matching of size at least $k$. \\
\end{abstract}

\maketitle
\section{Introduction}

We consider edge-colored  {\em simple} graphs.
A subgraph $H$ of such graph $G$ is {\em monochromatic} if every edge of $H$ is colored with the same color,
 and {\em rainbow} if no two edges of $H$ have the same color.
In the literature, a rainbow subgraph is also called totally multicolored, polychromatic, and heterochromatic.

In anti-Ramsey theory, for given $n$ and a graph $H$, the objective is to find the largest integer $k$
such that there is a coloring of $K_n$ using exactly $k$ colors that contains no rainbow copy of $H$.
The anti-Ramsey numbers and their relation to the Tur\' an numbers were first discussed by Erd\H os, Simonovits, and S\' os \cite{antiRamsey}.
Solutions to the anti-Ramsey problem are known for trees \cite{JiangWest}, matchings \cite{FujitaKanekoSchiermeyerSuzzuki}, and complete graphs \cite{Schiermeyer}, \cite{CompleteAntiRamsey} (see~\cite{antiRamseySurvey} for a more complete survey).
 R\" odl and Tuza proved there exist graphs $G$ with arbitrarily large girth
such that every proper edge coloring of $G$ contains a rainbow cycle \cite{R�dlTuza}.
 Erd\H os and Tuza asked  for which graphs $G$  there is a $d$ such that there is a rainbow copy of $G$ in any edge-coloring
of $K_n$ with exactly $|E(G)|$ colors such that
 for every vertex $v \in V(K_n)$ and every color $\alpha$, $v$ is the center of a monochromatic star with $d$ edges and color $\alpha$.
They found positive results for trees, forests, $C_4$, and $K_3$ and
found negative results for several infinite families of graphs~\cite{ErdosTuza}.

For $v \in V(G)$ and a coloring $\phi$ on $E(G)$, $\hat{d} (v)$ is the number of distinct colors on the edges incident to $v$.
This is called the \emph{color degree of $v$}.
The smallest color degree of all vertices in $G$ is the \emph{minimum color degree of $G$}, or $\hat{\delta} (G, \phi)$.
The largest color degree is $\hat{\Delta} (G, \phi)$.


Local anti-Ramsey theory seeks to find the maximum $k$ such that there exists a
coloring $\phi$ of $K_n$ that contains no rainbow copy of $H$ and $\hat{\delta} (K_n, \phi) \geq k$.

The topic of rainbow matchings has been well studied, along with a more general topic
of rainbow subgraphs (see~\cite{KanoLi} for a survey).
Let $r(G, \phi)$ be the size of a largest rainbow matching in a graph $G$ with edge coloring $\phi$.
In 2008, Wang and Li \cite{WangLi} showed that $r(G, \phi) \geq \left\lceil \frac{5  \hat{\delta} (G, \phi) - 3}{12}\right\rceil$
for every graph $G$ and conjectured that if $\hat{\delta} (G, \phi)\geq k \geq 4$ then $r(G, \phi) \geq \left\lceil \frac{ k }{2}\right\rceil$.
The conjecture is known to be tight for properly colored complete graphs.
LeSaulnier {\em et al.}~\cite{Illinois2010} proved that
$r(G, \phi) \geq \left\lfloor \frac{ k }{2}\right\rfloor$ for general graphs, and gave several conditions
sufficient for a rainbow matching of size $\left\lceil \frac{ k }{2}\right\rceil$.
In~\cite{KY2011}, the conjecture was proved in full. The only known extremal examples for the bound
have at most $k+2$ vertices.

Wang~\cite{Wang2011} proved that every properly edge-colored graph $(G,\phi)$ with $\delta(G, \phi)=k$ and $|V(G)|\geq 1.6k$
has a rainbow matching of size at least $3k/5$ and that every such triangle-free graph
has a rainbow matching of size at least $\left\lfloor 2k/3 \right\rfloor$. He also
asked if there is a function, $f(k)$, such that for every graph $G$
and proper edge coloring  $\phi$ of $G$ with $\hat{\delta}(G, \phi) \geq k$ and $|V(G)| \geq f(k)$,
 we have $r(G, \phi) \geq k$.
The bound on $r(G, \phi)$ is sharp for any properly $k$-edge-colored $k$-regular graph.

Diemunsch et al.~\cite{Denver} answered the question in the positive and
proved that $f(k) \leq 6.5 k$. 
Shortly thereafter, Lo~\cite{Lo}  improved the bound to $f(k) \leq 4.5 k$, and finally Diemunsch et al.~\cite{Denver2} combined the two manuscripts and improved the bound to $f(k) \leq \frac{98}{23} k$.
The largest matching in a graph with $n$ vertices contains at most $n/2$ edges, so $f(k) \geq 2k$.
By considering the relationship of Latin squares to edge-colored $K_{n,n}$, the lower bound can be improved to $f(k) \geq 2k + 1$ for even $k$.
This is the best known lower bound on the number of vertices required for both the properly edge-colored and general cases.

In this note we prove
an analogous result for arbitrary edge colorings of graphs.

\begin{theorem}\label{thm}
Let $G$ be an $n$-vertex graph and $\phi$ be an edge-coloring of $G$ with $n > 4.25 \d^2 (G, \phi)$.
Then $(G, \phi)$ contains a rainbow matching with at least $\d(G, \phi)$ edges.
\end{theorem}

Our result gives a significantly weaker bound on the order of $G$ than the bounds in~\cite{Denver2}
 but for a significantly wider class of edge-colorings.

Several ideas in the proof came from Diemunsch et al.'s paper~\cite{Denver}.
The full proof is presented in the next section.

\section{Proof of the Theorem}
Let $(G,\phi)$ be a counter-example to our theorem with the fewest edges in $G$.
For brevity, let $k := \hat{\delta}(G, \phi)$.
Since $(G,\phi)$ is a counter-example, $n:=|V(G)| > 4.25 k^2$.
The theorem is trivial for $k=1$, and it is easy to see that if $\hat{\delta}(G)=2$ and $(G, \phi)$ does not have a rainbow matching of size $2$, then $|V(G)| \leq 4$.
Therefore $k\ge 3$.

\begin{claim}\label{cl:starforest}
 Each color class in $(G,\phi)$ forms a star forest.
\end{claim}
\begin{proof} Suppose that the edges of color $\alpha$ do not form a star forest. Then there exists an edge
 $uv$ of color $\alpha$ such that an edge $ux$ and an edge $vy$ also are colored with $\alpha$ (possibly, $x=y$).
 Then the graph $G' = G-uv$ has fewer edges than $G$, but $\hat{\delta}(G', \phi)=k$.
By the minimality of $G$, $r(G', \phi) \geq k$.
But then $r(G, \phi) \geq k$, a contradiction.
\end{proof}

We will denote the set of maximal monochromatic stars of size at least $2$ by $\S$.
Let $E_0\subseteq E(G)$ be the set of edges not incident to another edge of the same color, i.e. the maximal monochromatic stars of size $1$.

\begin{claim}\label{cl:delta}
 For every edge $v_1v_2\in E(G)$, there is an $i\in\{1,2\}$, such that $\hat{d}(v_i)=k$ and $v_1v_2$ is the only edge of its color at $v_i$.
\end{claim}
\begin{proof}
 Otherwise, we can delete the edge and consider the smaller graph.
 \end{proof}

\begin{claim}\label{cl:delta2}
 All leaves $v\in V(G)$ of stars in $\S$ have $\hat{d}(v)=k$.
\end{claim}
\begin{proof}
This follows immediately from Claim~\ref{cl:delta}.
\end{proof}

For the sake of exposition, we will now direct all edges of our graph $G$.
With an abuse of notation, we will still  call the resulting directed graph $G$.
In every star in $\S$, we will direct the edges away from the center.
All edges in $E_0$ will be directed in a way such that
the sequence of color outdegrees in $G$, $\hat{d_0}^+\ge \hat{d_1}^+\ge\ldots\ge \hat{d_n}^+$ is lexicographically maximized.
Note that by Claim \ref{cl:starforest},
\begin{equation*}\label{rainbow in-edges}
\mbox{the set of edges towards $v$ forms a rainbow star, and so $d^-(v) \leq \hat{d}(v)$.}\tag{\texttt{I}}
\end{equation*}

Let $C$ be the set of vertices with non-zero outdegree and  $L:=V\setminus C$.
Let $\S^*\subseteq \S$ be the set of maximal monochromatic stars with at least two vertices in $L$, and let
$E_0^*\subseteq E_0\cup \S$ be the set of maximal monochromatic stars with exactly one vertex in $L$. 
For a color $\alpha$, let $E_H[\alpha]$ be the set of edges colored $\alpha$ in a graph $H$.
If there is no confusion, we will denote it by $E[\alpha]$.

\begin{claim}\label{cl:indeg}
For every $v\in V(G)$ with $\hat{d}(v)\geq k+1$, $d^-(v)=0$. In particular,
 $d^-(v)\le k$ for every $v\in V(G)$.
Moreover, for all $w \in L$, $d(w) = k$.
\end{claim}
\begin{proof}
Suppose that $\hat{d}(v)\geq k+1$, and let $w_iv$ be the edges directed towards $v$.
By Claim~\ref{cl:delta} and (\ref{rainbow in-edges}), $\hat{d}(w_i) = k$ and $w_iv\in E_0$ for all $i$.
Then $d^+(w_i)\le\hat{d}(w_i)=k$.
Reversing all edges $w_iv$ would increase the color outdegree of $v$ with a final value
larger than $k$ while decreasing the color outdegree of each $w_i$, which was at most $k$.
Hence the sequence of color outdegrees would lexicographically increase, a contradiction
to the choice of the orientation of $G$.

By the definition of $L$, if $w \in L$, then $d^+(w)=0$. So in this case by
the previous paragraph,  $k \le \hat{d}(w) \leq d^-(w) \le k$, which proves the second statement.
\end{proof}


\begin{claim}\label{cl:class}
 No color class in $(G,\phi)$ has more than $2k-2$ components.
\end{claim}
\begin{proof}
Otherwise, remove the edges of a color class $\alpha$ with at least $2k-1$ components,
and use induction to find a rainbow matching with $k-1$ edges in the remaining graph.
This matching can be incident to at most $2k-2$ of the components of $\alpha$, so there is at least one
component of $\alpha$ not incident to the matching, and we can pick any edge in this component to extend
the matching to a rainbow matching on $k$ edges.
\end{proof}

We consider three cases.
If $n > 4.25 k^2$, then at least one of the three cases will apply.
The first two cases will use greedy algorithms. 

\begin{case} $|\S^*|+\frac12|E_0^*|\ge 2.5k^2$.
\end{case}
For every $S\in \S^*$, assign a weight of $w_1(e)=1/|S\cap L|$ to each of the edges of $S$ incident to $L$.
Assign a weight of $w_1(e)=1/2$ to every edge $e\in E_0^*$.
Edges in $G[C]$ 
receive zero weight.
Let $G_0 \subset G$ be the subgraph of edges with positive weight.
For every set of edges $E'\subseteq E(G)$, let $w_1(E')$ be the sum of the weights of the edges in $E'$.
For every vertex, let $w_1(v)=\sum_{a\in N^+(v)}w_1(va) + \sum_{b\in N^-(v)}w_1(bv)$.
Note that $G_0$ is bipartite with partite sets $C$ and $L$ and that 
$w_1(e)\le 1/2$ for every edge $e \in E(G)$.
Furthermore,
\[
\frac12\sum_{v\in V(G)}w_1(v)=\sum_{e\in E(G)}w_1(e)=|\S^*|+\frac12|E_0^*|\ge 2.5k^2.
\]

 \begin{claim}\label{cl:Delta}
  For every $v\in V(G)$, $w_1(v) \le 2(k-1)$.
 \end{claim}
 \begin{proof} Suppose  $(G, \phi)$ has a vertex  $v$ with $w_1(v)>2(k-1)$.
Let $G'=G-v$.
Then $\hat{\delta}(G', \phi) \geq k-1$ and $|V(G')|=n-1> 4.25 (k-1)^2$.
By the minimality of $(G, \phi)$, the colored graph $(G', \phi)$ has a rainbow matching $M$ of size $k - 1$.
At most $k-1$ of the stars $v$ is incident to have colors appearing in $M$; each of them contributes a weight of at most $1$ to $w_1(v)$.
As $w_1(v)> 2(k-1)$, there are at least $2k-1$ edges incident to $v$ with colors not appearing in $M$.
At least one of these edges is not incident to $M$.
Thus $(G, \phi)$ has  a rainbow matching of size $k$, a contradiction.
\end{proof}

We propose an algorithm that will find a rainbow matching of size at least $k$.
For $i=1,2,\ldots$, at Step $i$:
\begin{enumerate}\setcounter{enumi}{-1}
\item  If $G_{i-1}$ has no edges or $i-1=k$, then stop.
	\item If a vertex $v$ of maximum weight has $w_1(v)  > 2(k-i)$ in $G_{i-1}$, then set $G_i = G_{i-1} - v$ and
go to Step $i+1$.
	\item If the largest color class $E[\alpha]$ of $G_{i-1}$ has at least $2(k-i)+1$ components, then set $G_i = G_{i-1} - E[\alpha]$ and
go to Step $i+1$.
	\item If $w_1(v) \leq 2(k-i)$ for all $v \in V(G_{i-1})$ and every color class has at most $2(k-i)$ components,
 then set $G_i = G_{i-1} - x - y - E[\phi(xy)]$ for some edge $xy \in E(G_{i-1})$. 
\end{enumerate}
We will refer to these as options $(1)$, $(2)$, and $(3)$ for Step $i$.
We call the difference in the total weight of the remaining edges between $G_{i-1}$ and $G_i$ the \emph{weight of Step $i$} or $W_1(i)$.
When both options $(1)$ and $(2)$ are possible, we will pick option $(1)$.
Option $(3)$ is only used when neither of options $(1)$ and $(2)$ are possible.

Let $G_r$ be the last graph created by the algorithm, i.e., $r = k$ or $G_r$ has no edges.
We will first show by reversed induction on $i$ that
\begin{equation*}\label{j21}
\mbox{ $G_i$ has a rainbow matching of size at least $r-i$}.\tag{\texttt{II}}
\end{equation*}
This trivially holds for $i=r$.
Suppose (\ref{j21}) holds for some $i$, and $M_i$ is a rainbow matching of size $r-i$ in $G_i$. 
If we used Option $(1)$ in Step $i$, then there is some edge $e \in E(G_{i-1})$ incident with $v$
that is not incident with $M_i$ and whose color does not appear on the edges of $M_i$, similarly to the proof of Claim~\ref{cl:Delta}.
If we used Option $(2)$ in Step $i$, then there is some component of $E_{G_{i-1}}[\alpha]$
that is not incident with $M_i$, and we let $e$ be an edge of that component.
If we used Option $(3)$ in Step $i$, then let $e = xy$.
In each scenario, $M_i+e$ is a rainbow matching of size $r-i+1$ in $G_{i-1}$.
This proves the induction step and thus (\ref{j21}). So, if $r=k$, then we are done.

Assume $r< k$.
Then the algorithm stopped because $E(G_{r+1})=\emptyset$.
This means that
\begin{equation}\label{f1}
 \sum_{i=1}^rW_1(i)=\sum_{e\in E(G)}w_1(e)\ge2.5k^2.\tag{\texttt{III}}
\end{equation}
We will show that this is not the case.
Suppose that at Step $i$, we perform Option $(3)$.
By the bipartite nature of $G_0$, we may assume that $y \in L$.
By Claim \ref{cl:indeg}, $w_1(y) -w_1(xy)\leq \frac{k-1}{2}$.
Because Options $(1)$ and $(2)$ were not performed at Step $i$, $w_1(x) + w_1(E_{G_{i-1}}[\phi(xy)]) \leq 4(k-i)$.
Therefore the weight of Step $i$ is at most $\frac{k-1}{2} + 4(k - i)  < 4.5k-4i$. 

By Claims \ref{cl:class} and \ref{cl:Delta}, Option $(3)$ is performed at Step $1$. If $W_1(i)< 4.5k-4i$ for all
$i$, then $\sum_{i=1}^rW_1(i)< \sum_{i=1}^r 4.5k-4i=4.5kr-2r(r+1)\leq 2.5k(k-1)$, a contradiction to (\ref{f1}).
Let $i$ be the first time that $W_1(i) \geq 4.5k - 4i$, and $j<i$ be the last time Option $(3)$ is performed prior to $i$.
By the choice of $i$, $W_1(a)<4.5k-4a$  when $a \leq j$.
Because Option $(1)$ and $(2)$ were not chosen at Step $j$, $W_1(i')\le 2(k-j)$ for each Step $i'$ such that $i' > j$ and Option $(1)$ or $(2)$ is used.
Note that by choice of $i$ and $j$, this bound applies for all steps between $j+1$ and $i$.
Furthermore, by the choice of $i$, $2(k-j) > 4.5k - 4i'-1$ for $i' > i$.
It follows that $W_1(b)\leq 2(k-j)$ for each $b>j$, and so
$$\sum_{a=1}^rW_1(a)
\leq \sum_{a=1}^{j}(4.5k - 4a)+2(k-j)(r-j)\leq 4.5kj-2j(j+1)+2(k-j)(k-1-j)
$$
$$= k(0.5j+2k-2)<2.5k^2,
$$
a contradiction to (\ref{f1}).

\begin{case}
 $|C|\ge 1.75 k^2$.
\end{case}
We will use a different weighting:
For every vertex $v\in C$ and outgoing edge $vw$, if
 $vw\in E_0$, we let $w_2(vw)=1/\hat{d}^+(v)$, where $\hat{d}^+(v)$ is the color outdegree of $v$, and if
$vw$ is in a star $S\in \S$, then we let $w_2(vw)=1/(\hat{d}^+(v)\|S\|)$.
For a vertex $v\in V(G)$, let $w^+(v)$ and $w^-(v)$ denote the accumulated weights of the outgoing and
incoming edges, respectively, and $w_2(v)=w^+(v)+w^-(v)$. By definition, $w^+(v)=1$ for each $v\in C$.
Then
\[
\sum_{e\in E(G)}w_2(e)=\sum_{v\in V(G)}w^-(v)=\sum_{v\in V(G)}w^+(v)=|C|\ge 1.75k^2.
\]

\begin{claim} \label{cl:case2 edge weight}
Let $uv$ be a directed edge in $G$ and $e$ an edge incident to $u$ that is not $uv$.
Then $w_2(e) \leq 1/2$.
\end{claim}
\begin{proof}
The result is easy if $e$ is in a monochromatic star with size at least $2$, so assume $e \in E_0$.
If $e$ is directed away from $u$, then $\hat{d}^+(u) \geq 2$ and the claim follows.
Suppose now that $e$ is directed towards $u$, say $e=wu$, and $w_2(e) = 1$. Then
$d^+(w)=1$, and reversing $e$ we obtain the orientation of $G$ where the outdegree of $w$
decreases from $1$ to $0$, and the outdegree of $u$ increases from $d^+(u) \geq 1$ to $d^+(u)+1$.
The new orientation has a  lexicographically larger outdegree sequence, which is a contradiction.
\end{proof}

\begin{claim}\label{cl:class2}
For every color  $\alpha$, we have $w_2(E[\alpha])\le 1.5(k-1)$.
\end{claim}
\begin{proof}
Otherwise, remove the edges of a color class $E[\alpha]$ with $w_2(E[\alpha])> 1.5(k-1)$,
and use induction to find a rainbow matching with $k-1$ edges in the remaining graph.
For every directed edge $vw\in M$, $v$ can be incident to a component of $E[\alpha]$ of weight at most $1/2$, and $w$
can be incident to a component of $E[\alpha]$ of weight at most $1$,  so there is at least one
component of $E[\alpha]$ not incident to the vertices of $M$, and we can pick any edge in this component to extend
$M$ to a rainbow matching of $k$ edges.
\end{proof}

We will use the following greedy algorithm: Start from $G$, and
at Step $i$, choose a color $\alpha$ with the minimum value $w_2(E[\alpha])>0$,
and pick any edge $e_i \in E[\alpha]$ of that color, and put it in the matching $M$, and then delete all edges of $G$
that are either incident to $e_i$ or have the same color as $e_i$.
Without loss of generality, we may assume that edge $e_i$ has color $i$.
If we can repeat the process $k$ times, we have found our desired rainbow matching, so assume that we run out of edges
after  $r<k$ steps, and call the matching we receive $M$.
Let $h\le k-1$ be the first step after which only edges with colors present in $M$ remain in $G_h$. Let $\beta$ be
a color not used in $M$ such that the last edges in $E[\beta]$ were deleted at Step $h$. Such $\beta$ exists, since
 $G$ has at least $k$ colors on its edges.

By Claim~\ref{cl:case2 edge weight}, one step can reduce the weight $w_2(E[\beta])$ by at most $1.5$.
It follows that $w_2(E[\beta])$ at Step $i\leq h$ is at most $1.5(h-i+1)$.
As we always pick the color with the smallest weight, the color $i\le h$ also had weight at most $1.5(h-i+1)$ when we deleted it in Step $i$.
Every color $i>h$ which appears in $M$ has weight at most $1.5(k-1)$ by Claim~\ref{cl:class2}.
Thus, the  total weight of colors in $M$ at the moment of their deletion is at most $1.5\sum_{i=1}^{h}{i} + 1.5(k-1)(k-1-h)$.

\begin{claim} \label{cl:case2 vertex weight}
For each vertex $v$, $w_2(v) \leq (k+1)/2$.
\end{claim}
\begin{proof}
Suppose there are two edges, $e_1$ and $e_2$, incident with $v$ such that $w_2(e_1) = w_2(e_2) = 1$.
By Claim~\ref{cl:case2 edge weight}, both edges are directed towards $v$ and are in $E_0$.
Consider the orientation of $G$ where the directions of $e_1$ and $e_2$ have been reversed.
Then the outdegree of $v$ has been increased by $2$, while the outdegree of two other vertices changed from $1$ to $0$.
This creates a lexicographically larger outdegree sequence,  a contradiction.

By Claim~\ref{cl:indeg}, if $\hat{d}(v)\geq k+1$, then $w_2(v)=1$. If $\hat{d}(v)=k$, then by the above $w_2(v) \leq 1+(k-1)/2$.
\end{proof}

If an edge $e$ has a color $\beta$ not in $M$ or has color $i\leq h$ but was deleted at Step $j$ with $j<i$,
then $e$ is incident to the edges $\{e_1,\ldots e_h\}$.
By Claim~\ref{cl:case2 vertex weight}, the total weight of such edges  is at most $2h(k+1)/2 $.

However, this is a contradiction because it implies
\[ |C| \le h(k+1) + \frac32\sum_{i=1}^{h}{i} +\frac32(k-1)(k-1-h) = \frac{3k^2}2-3k+\frac32 + \frac{3h^2}4-\frac{hk}2 + \frac{13h}4 < 1.75 k^2.
\]

\begin{case}
$|L|> |\S^*|+0.5|E_0^*|$.
\end{case}
We will introduce yet another weighting, now  of vertices in $L$.
For every star $S\in \S^*$,   add a weight of $1/|L\cap V(S)|$ to every vertex in $L\cap V(S)$.
For every edge $e\in E_0^*$, add a weight of $1/2$ to the vertex in $L\cap e$.
For every $v\in L$, let $w_3(v)$ be the resulting weight of $v$.

Since $\sum_{v\in L}w_3(v)= |\S^*|+0.5|E_0^*|<|L|$,
there is a vertex $v\in L$ with $w_3(v)<1$.
Let $S_1,S_2,\ldots S_{k}$ be the $k$ maximal monochromatic stars incident to $v$
 ordered so that $|L\cap V(S_i)|\le  |L\cap V(S_j)|$ for $1\le i<j\le k$  (where $S_1\in E_0$ is allowed).
Since $v\notin C$, all these stars have different centers and different colors.
Now we  greedily construct a rainbow matching $M$ of size $k$, using one edge from each $S_i$ as follows.
Start from including into $M$ the edge in $S_1$ containing $v$.
Assume that for $\ell\ge 2$, we have picked a matching $M$ containing one edge from each $S_i$ for $1\le i\le \ell-1$.
Since $w_3(v)<1$, we know that $|L\cap V(S_\ell)|> \ell$ for $\ell\ge 2$.
As every edge in $M$ contains at most one vertex in $L$, we can extend the matching with
an edge from the center of $S_\ell$ to an unused vertex in $L\cap V(S_\ell)$.

\bigskip
To finish the proof, let us check that at least one of the above cases holds.
If Cases 1 and 2 do not hold, then $|C|<1.75k^2$ and $|\S^*|+0.5|E_0^*|< 2.5k^2$.
Then, since $n>4.25k^2$, $|L|>4.25k^2-1.75k^2=2.5k^2$, and we have Case 3.
\qed

\end{document}